\documentclass[12pt]{amsart}

\newskip\stdskip                      
 \usepackage{amsfonts}
\usepackage{amsmath}
\usepackage{amssymb}
\usepackage{amstext}
\usepackage{amsthm}          

\usepackage{xypic}          

\usepackage[usenames,dvipsnames]{color}

\newcommand{\C}{\mathbb C}
\newcommand{\R}{\mathbb R}

\newcommand{\Z}{\mathbb Z} 
\newcommand{\N}{\mathbb N}
\newcommand{\PP}{\mathbb P}
\newcommand{\op}{\operatorname}

\newtheorem{thm}{Theorem}[section]
\newtheorem{lemma}[thm]{Lemma}
\newtheorem{cor}[thm]{Corollary}
\newtheorem{prop}[thm]{Proposition}

\theoremstyle{definition}
\newtheorem{dfn}[thm]{Definition}
\newtheorem{rem}[thm]{Remark}

\begin{document}

\title[Holomorphic spheres and $4$-dimensional symplectic pairs]{Holomorphic spheres and four-dimensional symplectic pairs}
\author{Gianluca Bande}
\address{Dipartimento di Matematica e Informatica, Universit\`a degli Studi di Cagliari, Cagliari, Italy}
\email{gbande@unica.it}
\author{Paolo Ghiggini}
\address{Laboratoire Jean Leray, CNRS and Universit\'e de Nantes, Nantes, France}
\email{paolo.ghiggini@univ-nantes.fr}

\thanks{The first author is partially supported by P.R.I.N. 2015 --  Real and Complex Manifolds: Geometry, Topology and Harmonic Analysis  -- Italy and by GESTA - Fondazione di Sardegna - Regione Sardegna. The second author is partially supported by the ANR grant ANR- 16-CE40-017 ``Quantact''.}

\date{\today}

\maketitle
\begin{abstract}
We classify four-dimensional manifolds endowed with symplectic pairs admitting embedded 
symplectic spheres with non-negative self-intersection, following the strategy of McDuff's 
classification  of rational and ruled symplectic four manifolds.
\end{abstract}

\section{Introduction}
A symplectic pair of type $(k, n-k)$ on a smooth manifold $M$ of dimension $2n$ \cite{BGK, BK, KM} is a pair of closed two-forms $(\omega, \eta)$ of constant and complementary ranks $2k$ and $2(n-k)$ respectively such that $\omega^{2k} \wedge \eta^{2(n-k)}$ is a volume form. To make the definition interesting, we will assume that $0 < k <n$. Then, when $M$ 
has dimension four  --- the case of interest in the present article --- a symplectic pair on $M$ 
can only be of type $(1,1)$. 

A symplectic pair gives rise to two symplectic forms
$$\Omega_+=\omega + \eta, \quad \Omega_-= \omega - \eta$$ 
on $M$ and on $(-1)^{n-p}M$, respectively, where $-M$ denotes the oriented manifold 
obtained by reversing the orientation of $M$. If $M$ has dimension four, a symplectic pair is equivalent to a pair of symplectic forms $(\Omega_+, \Omega_-)$ satisfying
$$\Omega_+^2 = - \Omega_-^2, \quad \Omega_+ \wedge \Omega_-=0.$$
In particular $M$ is symplectic for both orientations and $b_\pm(M) > 0$.

The kernels of $\omega$ and $\eta$ are integrable complementary distributions and therefore
integrate to a pair of transverse foliations ${\mathcal F}_\omega$ and ${\mathcal F}_\eta$, called {\it characteristic foliations},
such that
$$T{\mathcal F}_\omega=\ker \omega \quad \text{and} \quad T{\mathcal
  F}_\eta =\ker \eta.$$ 
See~\cite{BK} for example. Each form is symplectic on the leaves of the foliation induced 
by the other form and moreover ${\mathcal F}_\omega$ and ${\mathcal F}_\eta$ are 
symplectically orthogonal with respect to the symplectic forms $\Omega_{\pm}$.


Symplectic pairs appear naturally in the study of  Riemannian metrics for which the product 
of harmonic forms is still harmonic \cite{K} and in the investigation of the group cohomology 
of symplectomorphism groups \cite{KM}. 
In \cite{BK} several interesting examples and constructions are given, especially on closed 
four-manifolds. Among them we have manifold carrying Thurston geometries, flat symplectic 
bundles and Gompf's sum for symplectic pairs. Moreover it is proven that every 
$T^2$-bundle over $T^2$ carries a symplectic pair. Also in \cite{KM}, flat symplectic bundles 
are used to prove the existence of symplectic pairs on some closed four-manifold with non-vanishing signature.

At present the only known obstructions to the existence of a symplectic pair on closed 
manifolds are the obvious ones due to the cohomology classes determined by the symplectic 
pair, the existence of two transverse and complementary foliations and the fact that those 
manifolds are symplectic for both orientations.

This paper aims at being a first step towards the search of more refined obstructions to the 
existence of a symplectic pair and the classification of manifolds carrying such a structure. 
This is achieved by using the theory of $J$-holomorphic curves. To do that, we have to adapt 
the theory to our setting, meaning that we must consider almost complex structures 
for which the foliations have pseudoholomorphic leaves. 

Note that a symplectic pair provides no canonical way to choose an orientation over the other.
For this reason, every theorem we will state for $\Omega_+$ will also hold for $\Omega_-$.
 
Making use of some Bott-Baum formulas proved by Mu\~noz and Presas \cite{MP}, we first 
prove the following:
\begin{thm}\label{pippo}
Let $M$ be a closed $4$-manifold admitting a symplectic pair $(\omega, \eta)$. Then 
$(M, \Omega_+)$ is minimal.
\end{thm}
\begin{rem}\label{CP2 scoppiato}
Theorem \ref{pippo} implies that $\C P^2 \# \overline{\C P^2}$, which is the total space of the nontrivial $S^2$-bundle over $S^2$, does not admit a symplectic pair. However, it fulfils the cohomological properties required for the existence of a symplectic pair, has a splitting of
  the tangent bundle into rank-two subbundles and is symplectic for
  both orientations because it admits an orientation-reversing self-diffeomorphism.
\end{rem}

By applying the strategy of McDuff \cite{MC}, we can classify four-dimensional manifolds 
carrying symplectic pairs which admit embedded symplectic spheres with nonnegative 
self-intersection:
\begin{thm}\label{pluto} 
 Let $M$ be a closed four-manifold admitting a symplectic pair $(\omega, \eta)$ 
and let $S \hookrightarrow (M, \Omega_+)$ be a symplectically embedded sphere.
\begin{enumerate}
\item If $S \cdot S=0$, then $M$ is the total space of a flat
symplectic sphere bundle over a surface $\Sigma$, the fibres of  $M \to \Sigma$ are the 
leaves of one of the the foliations and $S$ is isotopic to a fibre.
\item If $S \cdot S >0$, then $M= S^2 \times S^2$ and $(\omega, \eta)$ is the symplectic 
pair induced by the product.
\end{enumerate}
\end{thm}

Finally, we prove a converse to Theorem \ref{pluto} by determining which sphere bundles
over a surface carry a symplectic pair.
\begin{thm}\label{paperino}
Let $M$ be the total space of a sphere bundle over a surface $\Sigma$. If the bundle is trivial
or $\Sigma$ has positive genus, then $M$ carries a symplectic pair such that the fibres are 
leaves of one of the characteristic foliations.
\end{thm}

We say that a four-manifold is {\em amphisymplectic} if it admits symplectic forms inducing opposite orientations. We have seen that four-manifolds admitting contact pairs are amphisymplectic, and it is natural to investigate how much of the main results of this article extend to this more general setting. For amphisymplectic four-manifolds a weaker version of the results of this article can be obtained via Seiberg-Witten theory. In fact, well-known results of Taubes, Kronheimer and Mrowka \cite{KM2, T} imply that if $M$ is a symplectic four-manifold with $b_+(M) >1$, then every homologically essential embedded sphere $S$ satisfies $S \cdot S<0$. Considering that the intersection product changes sign when the orientation of the manifold is reversed, we obtain the following.
\begin{prop}\label{amphisymplectic almost minimal}
If $M$ is an amphisymplectic four-manifold and $S \subset M$ and embedded sphere with $S \cdot S=-1$, then $b_-(M)=1$.
\end{prop}
We have seen in Remark~\ref{CP2 scoppiato} that $\C \PP^1 \# \overline{\C \PP}^1$ is amphisymplectic but does not admit a symplectic pair. This example shows that Proposition~\ref{amphisymplectic almost minimal} is not sufficient to recover Theorem~\ref{pippo}.

Combining Proposition \ref{amphisymplectic almost minimal} with McDuff's classification of rational and ruled symplectic manifolds \cite{MC}, we obtain the following corollary.
\begin{cor}\label{pluto}
Let $M$ be an amphisymplectic manifolds. If there exists an 
embedded symplectic sphere $S\subset M$ with $S \cdot S \ge 0$, then $M$ is an $S^2$-bundle over a surface $\Sigma$. Moreover, if $S\cdot S>0$, then $\Sigma=S^2$.
\end{cor}

This article is organised as follows. In Sections 
\ref{sec: aggiunzione} we will  discuss the 
version of the Bott-Baum formulas we will use. In Section \ref{sec: moduli} we will develop 
the technical details useful to adapt the theory of $J$-holomorphic curves to symplectic pairs. 
Section \ref{sec: main} contains the proofs of the main theorems. 

\section{Adjunction formulas}\label{sec: aggiunzione}
In this section we collect some results about $J$-holomorphic curves  in four-manifolds 
admitting $J$-holomorphic foliations. By a {\em $J$-holomorphic curve} we will mean 
a close, connected and embedded surface $S \subset M$ such that $J(TS)=TS$. Let $(M, 
{\mathcal F})$ be a four-manifold  endowed with a codimension two foliation. Throughout 
this section we will assume that $J$ is an almost complex structure on $M$ which preserves 
the tangent distribution $T{\mathcal F}$. 

Now we study how $J$-holomorphic curves which are not leaves of ${\mathcal F}$ intersect 
${\mathcal F}$. The starting point is the following lemma.
\begin{lemma}\label{tangenze discrete}
Let $S \subset M$ be a $J$-holomorphic curve which is not a leaf of 
${\mathcal F}$. Then the set ${\mathcal T}_S \subset S$ defined as
\[ {\mathcal T}_S = \{ x \in S : T_zS \subset T_z{\mathcal F} \} \]
is finite.
\end{lemma} 
\begin{proof}
We follow the proof of \cite[Lemma~2.4.3]{MCS2} very closely to show that every point
$x \in S$ has a neighbourhood which either is contained in a leaf, or intersects 
${\mathcal T}_S$ in a finite set. Since $S$ is compact and connected and ${\mathcal T}_S$ 
is closed, this will prove the lemma.

Given $x \in S \setminus int({\mathcal T}_S)$, after choosing coordinates $z=s+it$ in $S$ and 
$(w_1, w_2)$ in $M$ around $x$, we can assume that:  
\begin{itemize}
\item[(i)] $S$ around $x$ is parametrised by a smooth map $u \colon \Omega \to \C^2$ for 
$\Omega \subset \C$ an open neighbourhood of $0$ and $u(0)=0$,
\item[(ii)] the leaves of ${\mathcal F}$ are the planes with constant coordinate $w_2$, 
\item[(iii)] $u = (u_1, u_2)$ with $u_2$ not identically zero, and
\item[(iv)] $J(w_1, 0)=J_0$ for any $w_1 \in \C$, where $J_0$ is the canonical almost complex 
structure on $\C^2$.
\end{itemize}
Note that we consider here $\C^2$ as a real manifold with an almost complex structure $J \colon \C^2 \to GL(4, \R)$.

By condition  (iv) we can write $J(w_1, w_2)= J_0 + I(w_1, w_2)w_2$ for a smooth map
$I \colon \C^2 \to \op{Hom}_\R (\C, \op{End}_\R(\C^2))$. Using this decomposition, we 
can show that $u_i$, $i=1,2$, satisfies the equation
\begin{equation} \label{quasi CR}
\frac{\partial u_i}{\partial s} + J_0 \frac{\partial u_i}{\partial t} + C_iu_2 
\end{equation}
for a smooth function $C_i \colon \Omega \to \op{End}_\R(\C)$. If we apply $\partial_s - J_0 
\partial_t$ to the above equation for $i=2$ and estimate the terms of order zero and one, 
we obtain the inequality
\[ | \Delta u_2 | \le c( |u_2| + |\nabla u_2|), \]
which holds in a possibly smaller open neighbourhood of of $0$ in $\Omega$. Then 
Aronszajn's theorem \cite[Theorem~2.3.4]{MCS2} implies that the Taylor expansion of $u_2$ 
at $0$ is non trivial, and thus there exist real polynomials $(p_1, p_2)$ of degree $l >0$ such 
that $u_i(s,t)=p_i(s,t)+ o(|z|^l)$ and $p_2 \ne 0$ is homogeneous. Since $C_iu_2 =o(|z|^l)$, 
Equation
\eqref{quasi CR} implies that $p_1$ and $p_2$ are complex polynomials; in particular 
$p_2(z)=az^l$ for $a \in \C \setminus \{ 0 \}$.

Tangencies between $S$ and ${\mathcal F}$ are the same as critical points of $u_2$, so it 
will be enough to show that, in some neighbourhood of $0$, the function $u_2$ has no 
critical point other than (possibly) $0$ itself. In fact, $u_2'(z) = laz^{l-1} +o(|z|^{l-1})$ and 
therefore, for $|z| < \varepsilon$ sufficiently small, $|u_2'(z)-laz^{l-1}| < la|z|^{l-1}$. This 
implies that $u_2'(z) \ne 0$ if $0<|z|< \varepsilon$, and therefore proves the lemma.
\end{proof}

The following formulas, proved by Mu\~noz and Presas \cite{MP}, generalise on the one 
hand the intersection theory for $J$-holomorphic curves originally due to McDuff (see 
\cite[Appendix~E]{MCS2} for a comprehensive treatment), and on the other hand some 
results of Brunella \cite{Bru} in the context of holomorphic foliations on compact complex 
surfaces. Brunella, as well as Mu\~noz and Presas, consider also singular foliations. We will 
not need that level of generality.

Given an embedded $J$-holomorphic curve $S$ which is not contained in a leaf of 
$\mathcal{F}$, one can define an integer number $\sigma(x,S,\mathcal{F})$ for each 
point $x\in S$ which verifies the following properties:
\begin{enumerate}
\item $\sigma(x,S,\mathcal{F})\geq 0$, and
\item $\sigma(x,S,\mathcal{F})=0$ if and only if $S$ and $\mathcal{F}$ are tranverse 
at $x$.
\end{enumerate}
Note that $\sigma(x,S,\mathcal{F}) >0$ only at finitely many points by 
Lemma~\ref{tangenze discrete}.\footnote{It seems to us that the argument given in 
\cite{MP}, based on the local intersection of $S$ with the leaves of ${\mathcal F}$, is not 
sufficient.}

The number $\sigma(x,S,\mathcal{F})$ is defined as 
$$\sigma(x,S,\mathcal{F}) = I(x, S, {\mathcal F}_x)-1,$$
where $I(x, S, {\mathcal F}_x)$ is the local intersection number at $x$ between $S$ and the 
leaf of ${\mathcal F}$ passing through $x$ defined by McDuff; 
see~\cite[Appendix~E.2]{MCS2}.

The {\em tangency number} of $S$ with respect to $\mathcal{F}$ is defined as
$$
\Sigma_\mathcal{F}(S)=\sum_{x\in S} \sigma(x,S,\mathcal{F}).
$$

Denoting by $N_\mathcal{F}$ the normal bundle of $\mathcal{F}$, the value of the first 
Chern class of $N_\mathcal{F}$ on $S$ can be computed as follows.
\begin{prop}[{\cite[Lemma~4.2]{MP}}]\label{p:adjontion-1}
For a compact embedded $J$-holomorphic curve $S$ in $M$ which is not a leaf of 
$\mathcal{F}$ we have:
$$
<c_1(N_\mathcal{F}), S>=\chi (S)+\Sigma_\mathcal{F}(S).
$$
\end{prop}

Now we go back to the case of symplectic pairs on four-manifolds, where we have a pair of transverse foliations ${\mathcal F}_\eta$ and ${\mathcal F}_\omega$. 
\begin{prop}\label{p:adjontion-2}
Let $M$ be a four-manifold equipped with a pair of transverse codimension two foliations ${\mathcal F}_\eta$ and ${\mathcal F}_\omega$. If $J$ is an almost complex structure on $M$ preserving the distributions $T{\mathcal F}_\eta$ and $T{\mathcal F}_\omega$ and $S$ is a
$J$-holomorphic curve which is not a leaf of either ${\mathcal F}_\omega$ or ${\mathcal F}_\eta$, then 
\begin{equation} \label{main inequality}
S\cdot S \geq \chi (S).
\end{equation}
\end{prop}
\begin{proof}
The tangent bundle of $M$ splits as $TM=T\mathcal{F}_{\omega} \oplus T\mathcal{F}_\eta$. Then for the normal bundles $N\mathcal{F}_\omega$ and  $N\mathcal{F}_\eta$ we have $$N\mathcal{F}_\omega \cong T\mathcal{F}_\eta, \quad N\mathcal{F}_\eta \cong T\mathcal{F}_\omega.$$
Applying Proposition \ref{p:adjontion-1} we get:
\begin{align*}
<c_1(TM), S>=&<c_1(T\mathcal{F}_{\omega}), S>+<c_1(T\mathcal{F}_\eta), S>\\
=&<c_1(N\mathcal{F}_\eta), S>+<c_1(N\mathcal{F}_\omega), S>\\
=&2\chi(S)+\Sigma_{\mathcal{F}_\eta}(S)+\Sigma_{\mathcal{F}_\omega}(S).
\end{align*}
On the other hand, by the adjuntion formula we have:
$$
<c_1(TM), S>=\chi(S) +S\cdot S .
$$
Comparing the two formulas we obtain:
$$
S\cdot S=\chi(S) +\Sigma_\mathcal{G}(S)+\Sigma_\mathcal{G}(S).
$$
Since the indices of tangency are nonnegative, we obtain
$$
S\cdot S\geq \chi(S) ,
$$
as desired.
\end{proof}

\section{Moduli spaces of $J$-holomorphic spheres}
\label{sec: moduli}
In this section we prove a generic transversality result for $J$-holomorphic curves, when $J$ 
is compatible with a symplectic pair. We denote by ${\mathcal J}(\omega, \eta)$ the set of 
$\Omega_+$-compatible almost complex structures on $M$ which make the leaves of
${\mathcal F}_\omega$ and ${\mathcal F}_\eta$ $J$-holomorphic.
\begin{lemma}
 ${\mathcal J}(\omega, \eta)$ is nonempty and contractible.
\end{lemma}
\begin{proof}
Let $T{\mathcal F}_\omega$ and $T{\mathcal F}_\eta$ be the tangent distributions to 
${\mathcal F}_\omega$ and ${\mathcal F}_\eta$ respectively. They are symplectic 
sub-bundles of $TM$, and we denote by ${\mathcal J}({\mathcal F}_\omega)$ and 
${\mathcal J}({\mathcal F}_\eta)$ the sets of compatible almost complex structures on 
them, which are nonempty and contractible by \cite[Proposition~2.63]{MCS1}. Since 
$T{\mathcal F}_\omega$ and $T{\mathcal F}_\eta$ are symplectic orthogonal, there is a 
symplectic bundle isomorphism $TM \cong T{\mathcal F}_\omega \oplus T{\mathcal F}_\eta$ 
which induces a bijection between ${\mathcal J}(\omega, \eta)$ and 
${\mathcal J}({\mathcal F}_\omega) \times {\mathcal J}({\mathcal F}_\eta)$. This proves the 
lemma.
\end{proof}
Given $J \in {\mathcal J}(\omega, \eta)$, we say that a smooth map $u \colon S^2 \to M$ is 
{\em $J$-holomorphic} if 
\[ du \circ i = J \circ du, \]
where $i$ is the standard complex multiplication on $TS^2$ coming from the 
identification $S^2 \cong \C \PP^1$. For any $A \in H_2(M; \Z)$ we denote by 
$\widetilde{\mathcal M}(A, J)$ the space of $J$-holomorphic maps $u \colon S^2 \to M$ 
such that $u_*[S^2]=A$, and define the moduli space ${\mathcal M}(A, J)$ as the quotient 
of $\widetilde{\mathcal M}(A, J)$ by the group $PSL(2, \C)$ of holomorphic reparametrisations
 of $S^2$. The topology on $\widetilde{\mathcal M}(A, J)$ is the $C^\infty$-topology and the 
topology on ${\mathcal M}(A, J)$ is the quotient topology.

A $J$-holomorphic map $u \colon S^2 \to M$ is {\em multiply covered} if there is a
$J$-holomorphic map $v \colon S^2 \to M$ and a nontrivial holomorphic branched covering 
$\varphi \colon S^2 \to S^2$ such that $u = v \circ \varphi$. If $u$ is not multiply covered 
we say that it is {\em simple}. We denote by $\widetilde{\mathcal M}^s(A, J)$ the subset of
$\widetilde{\mathcal M}(A, J)$ consisting of simple maps, and by ${\mathcal M}^s(A, J)$ its 
quotient by holomorphic reparametrisations of $S^2$.

For every map $u \in W^{1,p}(S^2, M)$ with $p>2$ there is a {\em linearised 
Cauchy-Riemann operator} 
\[D_u \colon W^{1,p}(u^*TM) \to L^p(\overline{\op{Hom}}_J(TS^2, u^*TM)), \] 
where $\overline{\op{Hom}}_J(TS^2, u^*TM)$ denotes the bundle of anti-$\C$-linear 
homomorphisms from $TS^2$ to $u^*TM$; see \cite[Proposition~3.1.1]{MCS2}. By 
\cite[Theorem~C.1.10]{MCS2} $D_u$ is a Fredholm operator of index $2 \langle c_1(M), A 
\rangle +4$. 

\begin{dfn}\label{regularity 1}
An almost complex structure $J \in {\mathcal J}(\omega, \eta)$ is {\em regular} for $A \in 
H_2(A; \Z)$ if $D_u$ is surjective for every $u \in \widetilde{\mathcal M}^s(A, J)$. It is 
{\em regular} if it is regular for all $A \in H_2(M; \Z)$.
\end{dfn}
The following is a standard result in the theory of $J$-holomorphic maps: see 
\cite[Theorem~3.1.5(i)]{MCS2}.
\begin{thm}
If $J$ is regular for $A$, then
\begin{itemize}
\item ${\mathcal M}^s(A, J)$ is a smooth manifold of dimension  $2 \langle c_1(M), A \rangle 
-2$ if $\langle c_1(M), A \rangle \ge 1$, or
\item ${\mathcal M}^s(A, J) = \emptyset$ if $\langle c_1(M), A \rangle \le 0$.
\end{itemize}
\end{thm}
Regular almost complex structures are generic in the set of compatible
almost complex structures by \cite[Theorem~3.1.5(ii)]{MCS2}. However
we will need to work with almost complex structures of a more
restricted type, and therefore we will need to make some minor changes
to the statement and the proof of the basic transversality result. In
order to simplify the statement we introduce the following
terminology: we say that a property holds for a {\em generic}
$J \in {\mathcal J}(\omega, \eta)$ if it holds for every $J$ in a
countable intersection of open dense subsets.

Given $J \in {\mathcal J}(\omega, \eta)$, let $\overline{\op{Sym}}_J(T{\mathcal F}_*)$, 
for $* \in \{ \omega, \eta \}$, be the space of smooth sections $Y_*$ of  
$\op{End}(T{\mathcal F}_*)$ such that
\[ Y_* J_*+J_* Y_*=0 \quad \text{and} \quad Y_*^\dagger = Y_*, \]
where $J_*$ denotes the restriction of $J$ to ${\mathcal F}_*$ and $Y_*^\dagger$ is the adjoint of
$Y_*$ with respect to the metric on $T{\mathcal F}_*$ induced by $J_*$ and $\Omega_+$.
It is clear that, if $Y=(Y_\omega, Y_\eta) \in \overline{\op{Sym}}_J(T{\mathcal F}_\omega) 
\oplus \overline{\op{Sym}}_J(T{\mathcal F}_\eta)$, then $e^{Y}Je^{-Y} \in 
{\mathcal J}(\omega, \eta)$.

Given a sequence $\boldsymbol{\varepsilon}= (\varepsilon_n)_{n \in \N}$ with $\varepsilon_n 
\to 0$, we define Floer's $C_{\boldsymbol{\varepsilon}}$-space
$\overline{\op{Sym}}_{J, \boldsymbol{\varepsilon}}(T{\mathcal F}_*)$ as the set of all
$Y_* \in \overline{\op{Sym}}_J(T{\mathcal F}_*)$ such that
\[ \sum_{n=0}^{\infty} \epsilon_n \| Y_* \|_{C^n} < + \infty. \]
We refer to \cite[Section 6.3.2]{Aeb} and \cite[Section 4.4.1]{W2} for more details.

Floer's $C_{\boldsymbol{\varepsilon}}$-space is a separable Banach space and, for a suitable choice of
$\boldsymbol{\varepsilon}$, it contains bump sections with arbitrarily small support and 
arbitrary values at any point of $M$: see \cite[Lemma 4.52]{W2}. We fix an arbitrary almost 
complex structure $J_0 \in {\mathcal J}(\omega, \eta)$ and define 
\[{\mathcal J}_{\boldsymbol{\varepsilon}} = \{ e^Y J_0 e^{-Y} : Y \in 
\overline{\op{Sym}}_{J_0, \boldsymbol{\varepsilon}}(T{\mathcal F}_\omega) \oplus 
\overline{\op{Sym}}_{J_0, \boldsymbol{\varepsilon}}(T{\mathcal F}_\eta) \}. \]
Then ${\mathcal J}_{\boldsymbol{\varepsilon}}$ is a separable Banach manifold, admits a
continuous inclusion into ${\mathcal J}(\omega, \eta)$, and its tangent space at $J \in 
{\mathcal J}_{\boldsymbol{\varepsilon}}$ is 
\[T_J {\mathcal J}_{\boldsymbol{\varepsilon}} = \overline{\op{Sym}}_{J, 
\boldsymbol{\varepsilon}}(T{\mathcal F}_\omega) \oplus \overline{\op{Sym}}_{J, 
\boldsymbol{\varepsilon}}(T{\mathcal F}_\eta). \]

We denote by $\widetilde{\mathcal M}^{\not \subset}(A, J)$ the
subspace of {\em simple} maps in $\widetilde{\mathcal M}(A, J)$ which
are not contained in a leaf of ${\mathcal F}_{\omega}$ or
${\mathcal F}_\eta$.
\begin{prop}\label{prop: foliated regularity}
For a generic $J \in {\mathcal J}_{\boldsymbol{\varepsilon}}$, the linearised operator 
$D_u$ is surjective for every $J$-holomorphic map $u \in  
\widetilde{\mathcal M}^{\not \subset}(A, J)$.
\end{prop}
\begin{proof}
The proof will follow \cite[Section~3.2]{MCS2} very closely, except that we will use Floer's
$C_{\boldsymbol{\varepsilon}}$-spaces instead of spaces of $l$ times differentiable almost 
complex structures.

We denote by ${\mathcal B}$ the set of maps $u \in W^{1,p}(S^2, M)$, with $p>2$, 
such that $u_*[S^2]=A$. This is a Banach manifold whose tangent space at $u$ is 
$W^{1,p}(u^*TM)$. We define the Banach bundle ${\mathcal E} \to {\mathcal B} \times
{\mathcal J}_{\boldsymbol{\varepsilon}}$ whose fibre at $(u, J)$ is 
\[ {\mathcal E}_{u,J}= L^p(\overline{\operatorname{Hom}}_J(TS^2, u^*TM)),\]
where the complex multiplication on $u^*TM$ is induced by $J$.

The map ${\mathcal G}(u, J)= du + J \circ du \circ i$ defines a smooth section of
${\mathcal E} \to {\mathcal B} \times {\mathcal J}_{\boldsymbol{\varepsilon}}$. We consider 
the {\em universal moduli space}
\[ \widetilde{\mathcal M}(A) = \{ (u, J) \in {\mathcal B} \times {\mathcal J}_{\boldsymbol{\varepsilon}} : 
{\mathcal G}(u, J) =0 \} \]
and its subset $\widetilde{\mathcal M}^{\not \subset}(A)$ consisting of pairs $(u, J)$ such that
$u$ is simple and is not contained in a leaf of ${\mathcal F}_\omega$ or ${\mathcal F}_\eta$.
We will prove that $\widetilde{\mathcal M}^{\not \subset}(A)$ is a smooth Banach 
submanifold of ${\mathcal B} \times {\mathcal J}_{\boldsymbol{\varepsilon}}$. 

Let $D_{u, J}{\mathcal G} \colon T_u {\mathcal B} \times T_J 
{\mathcal  J}_{\boldsymbol{\varepsilon}} \to {\mathcal E}_{(u, J)}$ be the vertical differential 
of ${\mathcal G}$, i.e. the composition of $d_{(u, J)} {\mathcal G}$ with the projection to 
the tangent space of the fibres. Then we need to show that $D_{(u, J)}{\mathcal G}$ is 
surjective if $(u, J) \in \widetilde{\mathcal M}^{\not \subset}(A)$.

Given $(\xi, Y) \in T_u {\mathcal B} \times T_J {\mathcal J}_{\boldsymbol{\varepsilon}}$, 
we have
\[ D_{(u, J)}{\mathcal G}(\xi, Y)= D_u \xi + Y(u) \circ du \circ i. \]

Now we prove that $D_{(u, J)}{\mathcal G}$ is surjective for every $(u, J) \in \widetilde{\mathcal M}^{\not \subset}(A)$. Assume, to the contrary, that the image of $D_{(u, J)}{\mathcal G}$ has positive 
codimension. Since $D_u$ is a Fredholm operator, and therefore its image is closed 
and has finite codimension, the operator $D_{(u, J)}{\mathcal G}$ has closed image.
Then there exists a nontrivial $\zeta \in L^q(\overline{\operatorname{Hom}}_J(TS^2, 
u^*TM))$  such that
\begin{equation}\label{orazio}
 \int_{S^2} \langle D_u \xi, \zeta  \rangle =0 
\end{equation}
for all $\xi \in T_u{\mathcal B}= W^{1,p}(u^*TM)$ and
\begin{equation}\label{clarabella}
\int_{S^2} \langle Y(u) \circ du \circ i, \zeta \rangle =0
\end{equation}
for every $Y \in T_J {\mathcal J}_{\boldsymbol{\varepsilon}}$. In both equations $\langle \cdot, \cdot \rangle$ denotes a
pointwise scalar product on $T^{0,1}S^2 \otimes_{\C} u^*TM$ and the integral is computed
with respect to some volume form on $S^2$.

Elliptic regularity and Equation~\eqref{orazio} imply that
$\zeta$ is smooth. Therefore, if $(u, J) \in  \widetilde{\mathcal M}^{\not \subset}(A)$ and 
$\zeta \ne 0$, then by Lemma \ref{tangenze discrete} and \cite[Proposition~2.5.1]{MCS2}
 the set ${\mathcal U}$ of the points $z \in S^2$ such that
\begin{itemize}
\item[(i)] $\zeta(z) \ne 0$,
\item[(ii)] $u^{-1}(u(z))= \{ z \}$, and
\item[(iii)] $d_{z}u(v) \not \in T_{u(z)}{\mathcal F}_\omega \cup 
T_{u(z)}{\mathcal F}_\eta$ for every $v \in T_{z}S^2 \setminus \{ 0 \}$
\end{itemize}
is open and nonempty.
 The idea of the proof is that we compensate the smaller set of almost complex 
structures with a stronger somewhere injectivity property.

We write $du \circ i = \phi_\omega + \phi_\eta$, where $\phi_* \in 
\op{Hom}_J(TS^2, T{\mathcal F}_*)$  for 
$* \in \{ \omega, \eta \}$.
By \cite[Lemma~3.2.2]{MCS2} the maps 
\begin{equation} \label{anche basta} 
\overline{\op{Sym}}_J(T_{u(z)}{\mathcal F}_*) \to 
\overline{\op{Hom}}_J(T_{z}S^2, T_{u(z)}{\mathcal F}_*), \quad \mathbf{a}_* 
\mapsto \mathbf{a}_* \circ \phi_*(z), 
\end{equation}
for $* \in \{ \omega, \eta \}$, are surjective when $\phi_*(z) \ne 0$. 

Take $z_0 \in {\mathcal U}$. Since $\zeta(z_0) \ne 0$, by the surjectivity of the maps 
\eqref{anche basta} there is
\[\mathbf{a} = 
(\mathbf{a}_\omega, \mathbf{a}_\eta) \in  
\overline{\op{Sym}}_J(T_{u(z_0)}{\mathcal F}_\omega) \oplus 
\overline{\op{Sym}}_J(T_{u(z_0)}{\mathcal F}_\eta)\] 
such that
\[\langle \mathbf{a}  \circ d_{z_0} u \circ i, \zeta(z_0) \rangle >0.\] 
By (ii) there exist a neighbourhood ${\mathcal U}_0$ of $z_0$ in ${\mathcal U}$ and 
a  neighbourhood ${\mathcal V}$ of $u(z_0)$ in $M$ such that $u^{-1}({\mathcal V}) 
\subset {\mathcal U}_0$, and $Y \in 
\overline{\op{Sym}}_{J, \boldsymbol{\varepsilon}}(T{\mathcal F}_\omega) \oplus 
\overline{\op{Sym}}_{J, \boldsymbol{\varepsilon}}(T{\mathcal F}_\eta)$, with support 
in ${\mathcal V}$ and $Y(u(z_0))= \mathbf{a}$,  such that  
$\langle Y(u(z)) \circ d_{z} u \circ i, \zeta(z) \rangle >0$ for all 
$z \in {\mathcal U}_0$. Then
$$\int_{S^2}  \langle Y(u) \circ du \circ i, \zeta \rangle >0,$$ 
contradicting Equation~\eqref{clarabella}. This proves that
$\zeta(z_0)=0$, contradicting $z_0 \in {\mathcal U}$, and thus a section 
$\zeta \in L^q(\overline{\operatorname{Hom}}_J(TS^2, u^*TM))$ satisfaying
Equations \eqref{orazio} and \eqref{clarabella} vanishes everywhere.
 This proves that
$D_{(u, J)}{\mathcal G}$ is surjective whenever
$(u, J) \in \widetilde{\mathcal M}^{\not \subset}(A)$, and therefore
the universal moduli space $\widetilde{M}^{\not \subset}(A)$ is a
Banach manifold.

From now on, the proof procedes as in the proof of \cite[Theorem~3.1.5(ii)]{MCS2}: 
if $J$ is a regular value of the projection $\pi \colon \widetilde{\mathcal M}^{\not \subset}(A) 
\to {\mathcal  J}_{\boldsymbol{\varepsilon}}$, for every $u \in \widetilde{\mathcal M}^{\not \subset}(A, J)$ the
linearised Cauchy-Riemann operator $D_u$ is surjective. By Sard-Smale
theorem, the regular values of $\pi$ are generic in ${\mathcal J}_{\boldsymbol{\varepsilon}}$.
\end{proof}

\begin{rem}
The conclusion of Proposition~\ref{prop: foliated regularity} also holds for $J$-holomorphic 
maps from higher genus Riemann surfaces. The proof needs only the standard modifications 
to take into account the variations of the complex structure at the source.
\end{rem}

Now we consider $J$-holomorphic maps whose image is everywhere tangent to a leaf of a characteristic foliation. 
\begin{lemma}\label{spheres in leaves}
Let $u \colon S^2 \to M$ be a simple $J$-holomorphic map. If the image of $u$ is contained 
in a leaf of ${\mathcal F}_*$, $* \in \{ \eta, \omega \}$, then $u$ parametrises a leaf of 
${\mathcal F}_*$. 
\end{lemma}
\begin{proof}
Suppose then that the image of $u$ is contained in a leaf $F$. If $F$ is a noncompact leaf, 
then $u_*[S^2] =0$ in $H_2(M)$ because $H_2(F)=0$, and therefore $u$ is constant.
Then $F$ is compact and therefore $u \colon S^2 \to F$ is a holomorphic branched covering.
Since $u$ is simple, the covering has degree one, and therefore $u$ is an embedding.
\end{proof}

The Reeb stability theorem will play a crucial role in dealing with $J$-holomorphic spheres which are tangent to a characteristic foliation, so we recall its statement.

\begin{thm}[Reeb stability theorem -- see {\cite[Theorem 2.4.3]{CC}}] \label{teorema 
del nonno}
If $L$ is a compact leaf  of a foliated manifold $(M, \mathcal{F})$ and if $L$ is diffeomorphic
to a sphere $S^k$, $k \ge 2$, then there is a foliated neighbourhood $V \subset M$ 
containing $L$ such that $V \cong S^k \times D^{n-k}$ and ${\mathcal F}|_{V}$ is the 
foliation by spheres induced by the product.
\end{thm}

We will use the automatic transversality result of Hofer, Lizan and Sikorav to deal with 
$J$-holomorphic maps $u \colon S^2 \to M$ whose image in contained in a leaf of either
${\mathcal F}_\omega$ or ${\mathcal F}_\eta$. The following theorem is a reformulation of
\cite[Theorem~1]{HLS}.
\begin{thm} \label{thm: automatic transversality}
Let $u \colon S^2 \to M$ be a $J$-holomorphic map for some almost complex structure $J$. If
$u$ is an embedding and $u_*[S^2]=A$ with $A \cdot A \ge -1$, then $D_u$ is surjective.
\end{thm}

\begin{cor}\label{regularity 2}
A generic $J \in {\mathcal J}(\omega, \eta)$ is regular.
\end{cor}
\begin{proof}
First we observe that a simple $J$-holomorphic maps $u \colon S^2 \to M$ with values in a 
leaf of either ${\mathcal F}_\eta$ or ${\mathcal F}_\omega$ satisfies automatic transversality. 
In fact, by Lemma~\ref{spheres in leaves} $u$ parametrises a leaf $F$ and by Reeb's 
stability theorem \ref{teorema del nonno} $F \cdot F =0$. Then Theorem~\ref{thm: 
automatic transversality} implies that $D_u$ is surjective. This together with Proposition 
\ref{regularity 1} implies that a generic almost complex structure in 
${\mathcal J}_{\boldsymbol{\varepsilon}}$ is regular for $A$, for any 
$A \in H_2(M; \Z)$. Since  genericity is preserved by countable intersections, it 
follows that a generic almost complex structure in ${\mathcal J}_{\boldsymbol{\varepsilon}}$ 
is regular. In particular this implies that the almost complex structure $J_0$ used to define
${\mathcal J}_{\boldsymbol{\varepsilon}}$ can be approximated by regular almost complex 
structures in the $C^{\infty}$ topology. Since $J_0$ was chosen arbitrarily, this proves that
regular almost complex structures are dense in ${\mathcal J}(\omega, \eta)$. Finally, using an argument due to Taubes (and explained in detail in \cite[Section 4.4.2]{W2}) we 
conclude that a generic almost complex structure in ${\mathcal J}(\omega, \eta)$ is generic.
\end{proof}

\section{Proof of the main theorems}\label{sec: main}
In this section we prove the main theorems of the article.
\begin{proof}[Proof of Theorem~\ref{pippo}]
Let $S\hookrightarrow (M, \Omega_+)$ be an embedded symplectic sphere with 
$S \cdot S=-1$. By \cite[Theorem~5.1]{W} and Corollary~\ref{regularity 2} there exists 
a generic $J\in \mathcal{J}(\omega, \eta)$ for which $S$ is isotopic to the image of a 
$J$-holomorphic embedding $u \colon S^2 \to M$. From now on we denote by $S$ the image 
of $u$.

By the Reeb stability Theorem~\ref{teorema del nonno},  if $S$ were a leaf of either 
${\mathcal F}_\omega$ of ${\mathcal F}_\eta$, then it would satisfy $S\cdot S=0$.
Since $S$ is not a leaf, we are in position to apply Proposition \ref{p:adjontion-1}, which gives
$$S\cdot S\geq \chi(S)=2.$$
Then $S$ cannot have self-intersection $-1$ and therefore $M$ is minimal.
\end{proof}

\begin{proof}[Proof of Theorem~\ref{pluto}(i)]
By Theorem \ref{pippo} $(M, \Omega_+)$ is minimal. If $S$ is an embedded symplectic 
sphere in $(M, \Omega_+)$ with $S \cdot S=0$, then by \cite{MC} (see also \cite{W} for a 
more modern treatment) for every $\Omega_+$-compatible regular almost complex 
structure $J$, there is a fibration 
$$\pi \colon M \to \Sigma$$ 
over a surface $\Sigma$ whose fibres are $J$-holomorphic spheres which are isotopic to $S$. 
By Corollary~\ref{regularity 2} we can assume that $J \in {\mathcal J}(\omega, \eta)$.

Suppose now that $S$ is a fibre of $\pi$. Since it violates the inequality~\eqref{main 
inequality}, it must be a leaf of either ${\mathcal F}_\omega$ or ${\mathcal F}_\eta$. We 
will assume without loss of generality that it is a leaf of ${\mathcal F}_\omega$.
Let us consider the following subset of $\Sigma$:
$$
X=\{x\in \Sigma\, | \,\pi^{-1}(x)\; \text{is a leaf of $\mathcal{F}_\omega$}\}.
$$
Since $S$ is both a leaf and a fibre, $X$ is non empty. We 
shall prove that $X$ is open and closed. 

To prove that it is closed, let us consider a sequence $\{x_n\}$ in $X$ converging to $x \in 
\Sigma$, and $F_n=\pi^{-1} (x_n)$. Since $F_n$ is a leaf for all $n$, we have that 
$\omega_{|TF_n}=0$. By a limiting argument we obtain $\omega|_{TF}=0$,
which means that the fibre over $x$ is a leaf of ${\mathcal F}_\omega$.

Now we prove that $X$ is open. For $x\in X$, the preimage $F=\pi^{-1}(x)$ is both a 
fibre of $\pi$ and a leaf of $\mathcal{F}_\omega$. Since $F$ is a sphere, by Reeb's Stability 
Theorem~\ref{teorema del nonno} there it is an open foliated neighbourhood $U$ of $F$
where every leaf is a sphere. 
Let  $V$ be an open neighbourhood of $x$ in $\Sigma$ such that $\pi^{-1}(V)\subset U$.

Consider $y\in V$ and $z\in \pi^{-1}(y)=S_y$. Let $F_z$ be the leaf of $\mathcal{F}$ passing 
through $z$. We know that $F_z$ is a $J$-holomorphic sphere. By McDuff's positivity of 
intersection~\cite[Theorem~E.1.5]{MCS2}, we have either $F_z=S_y$ or $S_y 
\cdot F_z>0$ because $F_z$ and  $S_y$ intersect in $z$.

But both $S_y$ and $F_z$ have the same homology class of $F$ and then we have:
$$
F_z \cdot S_y =F\cdot F=0.
$$

We conclude from this that $F_z=S_y$, and therefore $y \in X$ for all $y\in V$. Then $X$ is 
open, so we have $X=\Sigma$.
\end{proof}

\begin{proof}[Proof of Theorem~\ref{pluto}(ii)]
The symplectic manifold $(M, \Omega_+)$ is minimal by Theorem~\ref{pippo}. 
If $S \cdot S >0$, then \cite[Corollary~1.6]{MC} implies that $(M, \Omega_+)$ is
symplectomorphic to either $\C\PP^2$ or $S^2 \times S^2$ with a product symplectic form.
Since $\C\PP^2$ carries no symplectic pair because $b_-(\C \PP^2)=0$, only the latter 
possibility remains. The same arguments of the proof of (i) applied to the spheres $S^2 \times 
\{ * \}$ and $\{ * \} \times S^2$ show that the foliations ${\mathcal F}_\omega$ and 
${\mathcal F}_\eta$ are given by the product structure. 
\end{proof}
Before proving Theorem \ref{paperino}, we recall two well known results.
\begin{lemma} \label{classification of sphere bundles}
For every closed, oriented surface $\Sigma$ there are exactly two oriented $S^2$-bundles 
over $\Sigma$ up to isomorphism, and they have non diffeomorphic total spaces.
\end{lemma}
%
%
The symplectic pairs in Theorem \ref{paperino} will be constructed via flat $SO(3)$-bundles.
Isomorphisms classes of flat $SO(3)$-bundles over a surface 
$\Sigma$ are in bijection with the conjugacy classes of representations $\rho \colon 
\pi_1(\Sigma) \to SO(3)$. 
\begin{lemma}\label{triviality criterion}
The flat bundle with holonomy representation $\rho \colon \pi_1(\Sigma) \to SO(3)$ is trivial
if and only if $\rho$ can be lifted to a representation $\tilde{\rho} \colon \pi_1(\Sigma) \to
SU(2)$.
\end{lemma}
Finally, we can prove Theorem \ref{paperino}.
\begin{proof}[Proof of Theorem {\ref{paperino}}]
A trivial $S^2$-bundle over a surface always admits the product symplectic pair, so we only 
need to consider nontrivial bundles over positive genus surfaces.
 
Let $\Sigma$ be a closed surface of positive genus. Since the standard area form of $S^2$ is $SO(3)$-invariant, if $\rho \colon \pi_1(\Sigma) \to SO(3)$ is a representation, then the total space $M$ of the flat $S^2$-bundle over $\Sigma$ with holonomy $\rho$ carries a symplectic pair by the construction given in Section 3 of \cite{BK}.


It remains to find a representation inducing the nontrivial  $S^2$-bundles over $\Sigma$.
First, we choose two $\widetilde{A}, \widetilde{B} \in SU(2)$ such that 
$$\widetilde{A}  \widetilde{B} = - \widetilde{B} \widetilde{A}.$$
To find such elements, we can identify $SU(2)$ with the quaternions of norm one and take 
$\widetilde{A}= \mathbf{i}$ and $\widetilde{B}= \mathbf{j}$. We denote by $A$ and $B$
the images of $\widetilde{A}$ and $\widetilde{B}$ in $SO(3)$. Then $AB=BA$. 

We choose generators $a_1, b_1, \ldots, a_g, b_g$ of $\pi_1(\Sigma)$ satisfying the relation 
$$[a_1, b_1] \ldots [a_g, b_g]=1$$
 and define $\rho: \pi_1(\Sigma) \to SO(3)$ such that
$\rho(a_1)=A$, $\rho(b_1)=B$, and $\rho(a_i)= \rho(b_i)=I$ for $i= 2, \ldots, g$. By 
construction $\rho$ does not lift to a representation in $SU(2)$, and therefore the associated
flat $S^2$-bundle is nontrivial by Lemma \ref{triviality criterion}.
\end{proof}

\bibliographystyle{amsplain}

\end{document}